\documentclass[12pt]{article}
\usepackage{amsmath}
\usepackage{amssymb}
\usepackage{amsthm}
\usepackage{verbatim}
\usepackage{mathrsfs}

\newcommand{\R}[0]{\mathbb R}

\newcommand{\Ds}[0]{\mathcal D}

\newtheorem{Th}{Theorem}[section]
\newtheorem{Lemma}{Lemma}[section]
\newtheorem{Prop}[Lemma]{Proposition}

\begin{document}

\title{On the well-posedness of the inviscid 2D Boussinesq equation}
\author{H. Inci}

\maketitle

\begin{abstract}
In this paper we consider the inviscid 2D Boussinesq equation on the Sobolev spaces $H^s(\R^2)$, $s > 2$. Using a geometric approach we show that for any $T > 0$ the corresponding solution map, $(u(0),\theta(0)) \mapsto (u(T),\theta(T))$, is nowhere locally uniformly continuous. 
\end{abstract}

\section{Introduction}\label{section_introduction}
The initial value problem for the inviscid 2d Boussinesq equation is given by
\begin{align}
\label{boussinesq}
\begin{split}
 u_t + (u \cdot \nabla) u &= -\nabla p + \left(\begin{array}{c} 0 \\ \theta \end{array}\right) \\
\theta_t + (u \cdot \nabla) \theta &= 0 \\
\operatorname{div} u &= 0 \\
u(0)=u_0, & \quad \theta(0)=\theta_0
\end{split}
\end{align}
where for $t \in \R, x \in \R^2$, $u(t,x) \in \R^2$ is the velocity of the flow, $\theta(t,x) \in \R$ the temperature and $p(t,x) \in \R$ the pressure. Note that for $\theta \equiv 0$ we recover the incompressible 2d Euler. Local well-posedness of \eqref{boussinesq} is known -- see \cite{dongo}. In particular we have for any $T > 0$ a continuous solution map 
\begin{eqnarray*}
 \Phi_T:U_T \subseteq H^s(\R^2;\R^2) \times H^s(\R^2) &\to& H^s(\R^2;\R^2) \times H^s(\R^2),\\
 (u_0,\theta_0) &\mapsto& (u(T),\theta(T))
\end{eqnarray*}
mapping the initial values contained in the domain of definition $U_T$ to the time $T$ values. Our main result is

\begin{Th}\label{th_nonuniform}
Let $s > 2$. For any $T > 0$ the solution map $\Phi_T$ is nowhere locally uniformly continuous on $U_T$.
\end{Th}

To prove Theorem \ref{th_nonuniform} we will use a geometric formulation of \eqref{boussinesq}. This formulation is just \eqref{boussinesq} in Lagrangian coordinates, i.e. for $u$ a solution to \eqref{boussinesq} we consider its flow map
\[
 \varphi_t = u \circ \varphi,\quad \varphi(0)=\mbox{id}
\] 
where $\mbox{id}$ is the identity map in $\R^2$. We thus can write the first equation in \eqref{boussinesq} as
\begin{equation}\label{ode_ansatz}
 \varphi_{tt}=(u_t + (u\cdot \nabla) u) \circ \varphi = -\nabla p \circ \varphi + \theta \circ \varphi
\end{equation}
It will turn out that the right handside can be expressed in $\varphi$ and $\theta_0$, which will give as a second order ODE. The functional space for $\varphi$ is the diffeomorphism group
\[
 \Ds^s(\R^2):=\{ \varphi:\R^2 \to \R^2 \;|\; \varphi-\mbox{id} \in H^s(\R^2;\R^2), \quad \det(d_x \varphi) > 0 \quad \forall x \in \R^2 \}
\]
This is a topological group under composition -- see \cite{composition} for the details. The geometric formulation will be essentially the same as in \cite{lagrangian}.\\
The second equation in \eqref{boussinesq} reads as
\[
 \theta \circ \varphi = \theta_0
\]
i.e. $\theta_0$ is transported by the flow. In particular we have at time $T > 0$
\[
 \theta(T) = \theta_0 \circ \varphi(T)^{-1}
\]
To establish the nonuniformity we will use this formula for $\theta(T)$ to get appropriate variations. The same method was used in \cite{euler,b_family,sqg,hyperelastic}.

\section{The geometric formulation}\label{section_geometric}

Throughout this section we assume $s > 2$. To replace $\nabla p$ we proceed as in \cite{lagrangian} and take the divergence in the first equation of \eqref{boussinesq}
\[
 \sum_{i,j=1}^2 \partial_i u_j \partial_j u_i = -\Delta p + \partial_2 \theta
\]
where we used $\operatorname{div}u=0$. By the divergence-freenes we have
\[
 \sum_{i,j=1}^2 \partial_i u_j \partial_j u_i = \sum_{i,j=1}^2 \partial_i \partial_j (u_i u_j)
\]
Using the cut-off $\chi(D)$, with $\chi$ the indicator function of the unit ball in $\R^2$ and $\chi(D)$ the corresponding Fourier multiplier, we can write
\begin{eqnarray*}
 -\nabla p &=& \nabla \chi(D) \Delta^{-1} \sum_{i,j=1}^2 \partial_i \partial_j (u_i u_j) + \nabla (1-\chi(D)) \sum_{i,j=1}^2 \partial_i u_j \partial_j u_i - \Delta^{-1} \nabla \partial_2 \theta \\
&=:& \nabla B(u,u) - \Delta^{-1} \nabla \partial_2 (\theta_0 \circ \varphi^{-1})
\end{eqnarray*}
Replacing in \eqref{ode_ansatz} the pressure term we can write down the initial value problem for $\varphi$ and $v:=\varphi_t$
\begin{equation}\label{ode}
 \partial_t \left(\begin{array}{c} \varphi \\ v \end{array}\right) = \left( \begin{array}{c} v \\ \nabla B(v \circ \varphi^{-1},v \circ \varphi^{-1}) \circ \varphi - [\Delta^{-1} \nabla \partial_2 (\theta_0 \circ \varphi^{-1})]\circ \varphi + \left(\begin{array}{c} 0 \\ \theta_0\end{array}\right) \end{array} \right)
\end{equation}
with initial values $\varphi(0)=\mbox{id}$ and $v(0)=u_0$ and $\theta_0$ is treated as a parameter. At this point it is not clear whether $\operatorname{div}u=0$ is respected. This will be addressed below. First we show 
\begin{Prop}\label{prop_analytic}
The right handside in \eqref{ode} is analytic in the variables $\varphi,v$ and the parameter $\theta_0$. Thus \eqref{ode} is an ODE described by an analytic vector field on $\Ds^s(\R^2) \times H^s(\R^2;\R^2)$ depending analytically on the parameter $\theta_0$.
\end{Prop}

\begin{proof}
The analyticity of 
\[
 (\varphi,v) \mapsto \nabla B(v \circ \varphi^{-1},v \circ \varphi^{-1}) \circ \varphi 
\]
was established in \cite{lagrangian}. In \cite{sqg} the analyticity of
\[
 (\varphi,\theta_0) \mapsto (\mathcal R_k (\theta_0 \circ \varphi^{-1}) )\circ \varphi
\]
was established where $\mathcal R_k=\frac{\partial_k}{(-\Delta)^{1/2}}, k=1,2$ is the Riesz operator. By writing for $k=1,2$
\[
 (-\Delta^{-1} \partial_k \partial_2 (\theta_0 \circ \varphi^{-1})) \circ \varphi = [\mathcal R_k \left( \left((\mathcal R_2 (\theta_0 \circ \varphi^{-1}))\circ \varphi \right) \circ \varphi^{-1} \right)] \circ \varphi 
\]
we see that
\[
 (\varphi,\theta_0) \mapsto (-\Delta^{-1} \nabla \partial_2 (\theta_0 \circ \varphi^{-1})) \circ \varphi
\]
is analytic. Alternatively one can use the splitting $\chi(D)+(1-\chi(D))$ to show this. Altogether we see that \eqref{ode} is analytic in $(\varphi,v,\theta_0)$.
\end{proof}

By ODE theory (see e.g. \cite{lang}) we get for every initial value $\varphi(0)=\mbox{id}, v(0)=u_0$ and $\theta_0$ a solution $(\varphi,v)$ on some time interval $[0,T]$. Defining $u:=\varphi_t \circ \varphi^{-1}$ we see that $u$ solves
\[
 u_t + (u \cdot \nabla) u = \nabla B(u,u) - \Delta^{-1} \nabla \partial_2(\theta_0 \circ \varphi^{-1}) + \left( \begin{array}{c} 0 \\ \theta_0 \circ \varphi^{-1} \end{array}\right)
\]
Taking the divergence we see that the $\theta_0$ terms vanish and we are left with
\[
 \partial_t \operatorname{div} u = \chi(D) \left(2 (u \cdot \nabla) \operatorname{div} u + (\operatorname{div}u)^2 \right) - (u \cdot \nabla) \operatorname{div} u
\]
By doing the same calculations as in \cite{lagrangian} we can estimate
\[
 \partial_t ||\operatorname{div} u||_{L^2}^2 \leq C ||\operatorname{div}u||_{L^2}^2
\]
showing that $\operatorname{div}u=0$ is preserved if it is initially zero. Hence for $u_0 \in H_\sigma^s(\R^2;\R^2)$, the space of divergence-free $H^s$ vector fields, and $\theta_0 \in H^s(\R^2)$ we get a solution $\varphi$ to \eqref{ode} on some time interval $[0,T]$. By the considerations above we see that
\[
 u(t):=\varphi_t(t) \circ \varphi(t)^{-1}, \quad \theta(t):=\theta_0 \circ \varphi(t)^{-1},\quad t \in [0,T]
\]
solves the first equation in \eqref{boussinesq}. Furthermore $\operatorname{div}u(t)=0$ is preserved and we have trivially the second equation in \eqref{boussinesq}. By the local wellposedness of ODEs we thus have (see \cite{lagrangian} how to prove the uniqueness part)
\begin{Th}\label{th_lwp}
Let $s > 2$. The inviscid Boussinesq system \eqref{boussinesq} is locally wellposed in $H^s$.
\end{Th}

\section{Nonuniform dependence}\label{section_nonuniform}

In this section we prove the main theorem. We will argue in the line of \cite{sqg}. Note that the system \eqref{boussinesq} has the scaling $u_\lambda=\lambda u(\lambda t,x),\theta_\lambda=\lambda^2 \theta(\lambda t,x)$, i.e. for $u,\theta$ a solution $u_\lambda,\theta_\lambda$ is also a solution. We denote by $\Phi$ the time 1 solution map and $U \subset H_\sigma^s(\R^2;\R^2) \times H^s(\R^2)$ its domain of definition, i.e. for $(u_0,\theta_0) \in U$
\[
 \Phi(u_0,\theta_0)=\left(\begin{array}{c} \Phi^{(1)}(u_0,\theta_0) \\ \Phi^{(2)}(u_0,\theta_0) \end{array}\right) := \left( \begin{array}{c} u(1) \\ \theta(1) \end{array}\right)
\] 
By the scaling property we see that the time $T$ solution map is given by
\[
 \Phi_T(u_0,\theta_0)=\left( \begin{array}{c} \Phi_T^{(1)}(u_0,\theta_0) \\ \Phi_T^{(2)}(u_0,\theta_0) \end{array}\right) = \left(\begin{array}{c} \frac{1}{T} \Phi^{(1)}(Tu_0,T^2\theta_0) \\ \frac{1}{T^2} \Phi^{(2)}(Tu_0,T^2 \theta_0) \end{array}\right)
\]
Thus Theorem \ref{th_nonuniform} will follow from
\begin{Prop}\label{prop_nonuniform}
The map $\Phi$ is nowhere locally uniformly continuous on $U$.
\end{Prop}

For \eqref{ode} with initial values $\varphi(0)=\mbox{id}$, $v(0)=u_0$ and $\theta_0$ we denote for $t \geq 0$
\[
 \Psi^t(u_0,\theta_0)=\varphi(t)
\]
i.e. the time $t$ value of the $\varphi$-component. Further we introduce $\Psi \equiv \Psi^1$ for simplicity. Note that $\Psi^t(u_0,\theta_0)$ is analytic in its arguments. We have by the initial condition
\[
 \left. \partial_t \right|_{t=0} \Psi^t(u_0,\theta_0) = u_0
\]
By the scaling property of \eqref{boussinesq} we have the identity
\[
 \Psi^t(u_0,\theta_0) = \Psi(t u_0,t^2 \theta_0)
\]
So taking the $t$-derivative at $t=0$ we get
\[
 d_{(0,0)} \Psi (u_0,0) = u_0
\]
which is the differential of $\Psi$ at the point $(0,0) \in H^s(\R^2;\R^2) \times H^s(\R^2)$ in direction of $(u_0,0)$. With this we have

\begin{Lemma}\label{lemma_dense}
Let $U \subseteq H_\sigma^s(\R^2;\R^2) \times H^s(\R^2)$ be the domain of definition of the $\Psi$ (resp. $\Phi$). There is subset $S \subseteq U$ which is dense in $U$, the elements of which are of the form $(u_0,\theta_0)$ with $\theta_0$ compactly supported and it has the property
\begin{align*}
& \forall w_0 \in S \quad \exists w=(u,0) \in H_\sigma^s(\R^2;\R^2) \times H^s(\R^2), x^\ast \in \R^2 \mbox{ with } (d_{w_0} \Psi (w))(x^\ast) \neq 0\\
& \mbox{and where the distance of } x^\ast \mbox{ to the support of } \theta_0 \mbox{ is bigger than } 2.
\end{align*}
\end{Lemma}

\begin{proof}
We know that that compactly supported functions are dense in $H^s(\R^2)$. Take an arbitrary $w_0=(u_0,\theta_0) \in U$ with a compactly supported $\theta_0$ and $u \in H_\sigma^s(\R^2;\R^2)$ with $u(x^\ast) \neq 0$ for some $x^\ast \in \R^2$ with a distance of more than 2 to the support of $\theta_0$. Consider now
\[
 \gamma:t \mapsto \left(d_{t w_0} \Psi(u,0)\right)(x^\ast)
\]
which is an analytic map with $\gamma(0)=u(x^\ast) \neq 0$. Therefore we have a sequence $t_n \uparrow 1$ with $\gamma(t_n) \neq 0$. Hence we can put all the $t_n w_0$ into $S$. This shows the claim. 
\end{proof}

With this preparations we can now prove Proposition \ref{prop_nonuniform}.

\begin{proof}[Proof of Proposition \ref{prop_nonuniform}]
It will be enough to prove that for any $w_0 \in S$ from Lemma \ref{lemma_dense} and any $R > 0$ we have
\[
 \left. \Phi \right|_{B_R(w_0)} : w_0 \mapsto \Phi(w_0)
\]
is not uniformly continuous. Here $B_R(w_0)$ is the ball of radius $R$ with center $w_0=(u_0,\theta_0)$, i.e.
\[
 B_R(w_0)=\{ w=(u,\theta) \in H^s(\R^2;\R^2) \times H^s(\R^2) \;|\; ||w||_s < R \}
\]
where $||w||_s:=\max\{||u-u_0||_s,||\theta - \theta_0||_s\}$. So we take an arbitrary $w_0=(u_0,\theta_0) \in S$. Denote by $\varphi_\bullet = \Psi(w_0)$. By Lemma \ref{lemma_dense} there is $w^\ast=(u^\ast,0) \in H^s_\sigma(\R^2;\R^2) \times H^s(\R^2)$ and $x^\ast \in \R^2$ with
\[
 |\left(d_{w_0}\exp(w^\ast)\right)(x^\ast)| > m ||w^\ast||_s 
\] 
for some $m > 0$. By continuity we can choose $R_1 > 0$ with
\begin{equation}\label{below}
 \frac{1}{C_1} ||\theta||_s \leq ||\theta \circ \varphi^{-1}||_s \leq C_1 ||\theta||_s
\end{equation}
for some $C_1 > 0$ and for all $\theta \in H^s(\R^2)$ and for all $\varphi \in \Psi\left(B_{R_1}(w_0)\right)$ -- see \cite{composition} for the justification. By the choice of $x^\ast$ we have for $w_0=(u_0,\theta_0)$
\[
 d:=\operatorname{dist}\left(\varphi_\bullet(\operatorname{supp}\theta_0),\varphi_\bullet(B_1(x^\ast))\right) > 0
\]
where $\operatorname{dist}$ denotes the distance function for sets in $\R^2$, $\operatorname{supp}$ denotes the support of functions and $B_1(x^\ast)$ is the unit ball in $\R^2$. Further we introduce
\[
 K_1:=\{ x \in \R^2 \;|\; \operatorname{dist}\left(x,\varphi_\bullet(\operatorname{supp}\theta_0)\right) \leq d/4 \}
\]
and
\[
 K_2:=\{ x \in \R^2 \;|\; \operatorname{dist}\left(x,\varphi_\bullet(B_1(x^\ast))\right) \leq d/4 \}
\]
With this we choose $0 < R_2 \leq R_1$ with
\begin{equation}\label{lipschitz}
 |\varphi(x)-\varphi_\bullet(x)| \leq d/4 \mbox{ and } |\varphi(x)-\varphi(y)| < L |x-y|
\end{equation}
for some $L > 0$ and for all $x,y \in \R^2$ and for all $\varphi \in \Psi\left(B_{R_2}(w_0)\right)$. This is possible by the Sobolev imbedding of $H^s$ into $C^1$, i.e
\[
 ||f||_{C^1} \leq C ||f||_s,\quad \forall f \in H^s(\R^2)
\]
for some $C > 0$. This ensures in particular that
\[
 \varphi(\operatorname{supp}\theta_0) \subseteq K_1 \quad \mbox{and} \quad
 \varphi(B_1(x^\ast)) \subseteq K_2
\]
for all $\varphi \in \Psi(B_{R_2}(w_0))$. To motivate the next step consider the Taylor expansion of $\Psi$
\[
 \Psi(w_0+h)=\Psi(w_0)+d_{w_0}\Psi(h) + \int_0^1 (1-t) d_{w_0 + t h}^2 \Psi(h,h) \;dt
\]
where $d_{w}^2 \Psi$ is the second derivative. With this we choose $0 < R_\ast \leq R_2$ in such a way that we have
\[
 ||d_w^2 \Psi(h_1,h_2)||_s \leq K ||h_1||_s ||h_2||_s
\]
and
\[
 ||d_{w_1}^2\Psi(h_1,h_2)-d_{w_2}^2\Psi(h_1,h_2)||_s \leq K ||w_1-w_2||_s ||h_1||_s ||h_2||_s
\]
for all $w,w_1,w_2 \in B_{R_\ast}(w_0)$ and $h_1,h_2 \in H^s(\R^2;\R^2) \times H^s(\R^2)$. Finally we choose $0 < R_\ast \leq R_2$ with
\begin{equation}\label{R_choice}
 \max\{C K R_\ast,C K R_\ast^2\} < m/8
\end{equation}
with the $C$ from the Sobolev imbedding. The goal is now to construct for every $0 < R \leq R_\ast$ a pair of sequences of initial values
\[
 (w_0^{(n)})_{n \geq 1},(\tilde w_0^{(n)})_{n \geq 1} \subseteq B_R(w_0)
\]
with $\lim_{n \to \infty} ||w_0^{(n)}-\tilde w_0^{(n)}||_s=0$ whereas
\[
 \limsup_{n \to \infty} ||\Phi(w_0^{(n)})-\Phi(\tilde w_0^{(n)})||_s > 0
\]
which would show our claim. So fix now $0 < R \leq R_\ast$. The first sequence we choose as
\[
 w_0^{(n)}=w_0 + \left(\begin{array}{c} 0 \\ \theta^{(n)} \end{array}\right) 
\]
with some $\theta^{(n)} \in H^s(\R^2)$ with $||\theta^{(n)}||_s = R/2$ and $\operatorname{supp}\theta^{(n)} \subseteq B_{r_n}(x^\ast)$ where
\[
 r_n=\frac{m}{8nL} ||w^\ast||_s
\]
Thus the support of $\theta^{(n)}$ vanishes for $n \to \infty$, but its ''mass'' remmains constant. The second sequence we get by a perturbation of the first
\[ 
 \tilde w_0^{(n)} = w_0^{(n)} + \frac{1}{n} w^\ast = w_0 + \left(\begin{array}{c} 0 \\ \theta^{(n)} \end{array}\right) + \frac{1}{n} \left(\begin{array}{c} u^\ast \\ 0 \end{array}\right) 
\]
We clearly have for $n \geq N$ with $N$ large enough
\[
 (w_0^{(n)})_{n \geq 1},(\tilde w_0^{(n)})_{n \geq 1} \subseteq B_R(w_0) \mbox{ and } \operatorname{supp} \theta^{(n)} \subseteq B_1(x^\ast)
\]
Further by construction $||w_0^{(n)}-\tilde w_0^{(n)}||_s=||\frac{1}{n}w^\ast||_s \to 0$ as $n \to \infty$. It will be enough to prove 
\[
 \limsup_{n \to \infty} ||\Phi^{(2)}(w_0^{(n)})-\Phi^{(2)}(\tilde w_0^{(n)})||_s > 0
\]
for $\Phi^{(2)}$, the $\theta$-component of $\Phi$. We introduce for $n \geq N$
\[
 \varphi^{(n)}=\Psi(w_0^{(n)}) \quad \mbox{and} \quad \tilde \varphi^{(n)}=\Psi(\tilde w_0^{(n)})
\]
We thus have
\[
 \Phi^{(2)}(w_0^{(n)})=(\theta_0+\theta^{(n)}) \circ (\varphi^{(n)})^{-1} \quad \mbox{and}
\quad \Phi^{(2)}(\tilde w_0^{(n)})=(\theta_0+\theta^{(n)}) \circ (\tilde \varphi^{(n)})^{-1}
\]
So we consider
\begin{multline*}
 ||(\theta_0+\theta^{(n)}) \circ (\varphi^{(n)})^{-1}-(\theta_0+\theta^{(n)}) \circ (\tilde \varphi^{(n)})^{-1}||_s\\
=||\left(\theta_0 \circ (\varphi^{(n)})^{-1} - \theta_0 \circ (\tilde \varphi^{(n)})^{-1}\right) + \left(\theta^{(n)} \circ (\varphi^{(n)})^{-1}-\theta^{(n)} \circ (\tilde \varphi^{(n)})^{-1}\right)||_s
\end{multline*}
Note that the two terms in the first bracket are supported in $K_1$ and the latter two terms in $K_2$. Thus we can seperate the expressions (see \cite{sqg} for the details on this) and it suffices to establish
\[
 \limsup_{n \to \infty} ||\theta^{(n)} \circ (\varphi^{(n)})^{-1}-\theta^{(n)} \circ (\tilde \varphi^{(n)})^{-1}||_s > 0
\]
Note that the terms are supported in
\[
 \overline{\varphi^{(n)}(B_{r_n}(x^\ast))} \mbox{ resp. } \overline{\tilde \varphi^{(n)}(B_{r_n}(x^\ast))}
\]
To show that we can also seperate these expressions we estimate the difference $\varphi^{(n)}-\tilde \varphi^{(n)}$. To do this we use the Taylor expansion. With $h^{(n)}:=(0,\theta^{(n)})$ it reads as
\begin{multline*}
 \varphi^{(n)} = \Psi(w_0^{(n)}) = \Psi(w_0+h^{(n)}) = \\
\Psi(w_0) + d_{w_0}\Psi(h^{(n)}) + \int_0^1 (1-t) d_{w_0 + t h^{(n)}}\Psi(h^{(n)},h^{(n)}) \;dt
\end{multline*}
resp.
\begin{multline*}
 \varphi^{(n)} = \Psi(\tilde w_0^{(n)}) = \Psi(w_0+h^{(n)}+\frac{1}{n}w^\ast) = \\
\Psi(w_0) + d_{w_0}\Psi(h^{(n)}+\frac{1}{n}w^\ast) + \int_0^1 (1-t) d_{w_0 + t h^{(n)}+t \frac{1}{n}w^\ast}\Psi(h^{(n)}+\frac{1}{n}w^\ast,h^{(n)}+\frac{1}{n}w^\ast) \;dt
\end{multline*}
Taking the difference we get
\[
 \tilde \varphi^{(n)} - \varphi^{(n)} = d_{w_0}\Psi(\frac{1}{n} w^\ast) + I_1 + I_2 + I_3
\]
where 
\[
 I_1 = \int_0^1 (1-t) \left(d_{w_0 + t h^{(n)}+t \frac{1}{n}w^\ast}\Psi(h^{(n)},h^{(n)})-d_{w_0 + t h^{(n)}}\Psi(h^{(n)},h^{(n)}) \right) \;dt
\]
and
\[
 I_2 = 2 \int_0^1 (1-t) d_{w_0 + t h^{(n)}+t \frac{1}{n}w^\ast}\Psi(h^{(n)},\frac{1}{n}w^\ast)
\]
and
\[
 I_3 = \int_0^1 (1-t) d_{w_0 + t h^{(n)}+t \frac{1}{n}w^\ast}\Psi(\frac{1}{n}w^\ast,\frac{1}{n}w^\ast)
\]
We have the estimates
\[
 ||I_1||_s \leq K ||\frac{1}{n}w^\ast|| ||h^{(n)}||_s^2 = \frac{K R^2}{4n} ||w^\ast||_s
\]
and
\[
 ||I_2||_s \leq 2 K ||h^{(n)}||_s ||\frac{1}{n}w^\ast||_s = \frac{K R}{n} ||w^\ast||_s
\]
and 
\[
 ||I_3||_s \leq K ||\frac{1}{n}w^\ast||_s^2 = \frac{K}{n^2} ||w^\ast||_s
\]
Thus by the choice of $R_ast$ with \eqref{R_choice} and the Sobolev imbedding
\[
 |I_1(x^\ast)| \leq \frac{C K R^2}{4n} ||w^\ast||_s \leq \frac{m}{8n} ||w^\ast||_s
\]
and
\[
 |I_2(x^\ast)| \leq \frac{C K R}{n} ||w^\ast||_s \leq \frac{m}{8n} ||w^\ast||_s
\]
and
\[
 |I_3(x^\ast)| \leq \frac{C K}{n^2} ||w^\ast||_s \leq \frac{m}{8n} ||w^\ast||_s
\]
where the last inequality holds for $n$ large enough. We thus have
\[
 |\tilde \varphi^{(n)}(x^\ast)-\varphi^{(n)}(x^\ast)| \geq |d_{w_0}\Psi(\frac{1}{n}w^\ast)(x^\ast)|-3 \frac{m}{8n} ||w^\ast||_s \geq \frac{m}{2n} ||w^\ast||_s
\]
By the Lipschitz property \eqref{lipschitz} and the choice of $r_n$ 
\[
 \overline{\varphi^{(n)}(B_{r_n}(x^\ast))} \subseteq B_{\frac{m}{8n} ||w^\ast||_s}(\varphi^{(n)}(x^\ast))
\]
resp
\[
 \overline{\tilde \varphi^{(n)}(B_{r_n}(x^\ast))} \subseteq B_{\frac{m}{8n} ||w^\ast||_s}(\tilde \varphi^{(n)}(x^\ast))
\]
Thus they are supported in balls of radius $\frac{m}{8n} ||w^\ast||_s$ whose centers are more than $\frac{m}{2n} ||w^\ast||_s$ apart. So we have a situation described in \cite{sqg} and thus we can write with some $C' > 0$
\begin{multline*}
 ||\theta^{(n)} \circ (\varphi^{(n)})^{-1}-\theta^{(n)} \circ (\tilde \varphi^{(n)})^{-1}||_s \geq  \\
C' (||\theta^{(n)} \circ (\varphi^{(n)})^{-1}||_s + ||\theta^{(n)} \circ (\tilde \varphi^{(n)})^{-1}||_s) \geq \frac{2C'}{C_1} ||\theta^{(n)}||_s =\frac{C'}{C_1} R
\end{multline*}
where we used \eqref{below}. This shows that we have $\tilde C$ with
\[
 \limsup_{n \to \infty} ||\Phi(w_0^{(n)})-\Phi(\tilde w_0^{(n)})||_s \geq \tilde C R
\]
whereas $||w_0^{(n)}-\tilde w_0^{(n)}||_s \to 0$ as $n \to \infty$. As $R \in (0,R_\ast)$ was arbitrary this proves the proposition.
\end{proof}

\bibliographystyle{plain}

\flushleft
\author{ Hasan Inci\\
EPFL SB MATHAA PDE \\
MA C1 627 (B\^atiment MA)\\ 
Station 8 \\
CH-1015 Lausanne\\
Schwitzerland\\
        {\it email: } {hasan.inci@epfl.ch}
}

\end{document}